\documentclass[12pt,a4paper,reqno]{amsart}
\usepackage{amssymb}
\usepackage{amscd}
\usepackage{enumerate}
\usepackage{graphicx}
\usepackage{siunitx}
\usepackage{tikz-cd}
\usepackage{color}
\usetikzlibrary{arrows}
\numberwithin{equation}{section}

\usepackage{mathabx}

\usepackage{mathtools}
\usepackage[tableposition=top]{caption}
\usepackage{booktabs,dcolumn}
\usepackage{hyperref}

\DeclareFontFamily{OT1}{rsfs}{}
\DeclareFontShape{OT1}{rsfs}{n}{it}{<-> rsfs10}{}
\DeclareMathAlphabet{\mathscr}{OT1}{rsfs}{n}{it}

\theoremstyle{plain}

\newtheorem{theorem}{Theorem}[section]

\newtheorem{lemma}[theorem]{Lemma}

\theoremstyle{definition}

\newtheorem{remark}[theorem]{Remark}

\newcommand\R{\mathbb{R}}

\newcommand\C{\mathbb{C}}

\newcommand{\on}{\operatorname}

\usepackage{algorithm, algorithmic}

\usepackage[colorinlistoftodos]{todonotes}

\begin{document}
	
	\title{If $\sum_n n! c_n z^n$ is entire and $c_n$ does not terminate, then $\sum_n c_n z^n$ has infinitely many zeros}
	
	\author{Alann Rosas}
	\address{UC Irvine Department of Mathematics, Irvine, CA 92617.}
	\email{alannr@uci.edu}
	
	
	
	\begin{abstract} We prove that if $\sum_n n! c_n z^n$ is entire and $c_n$ does not terminate, then $\sum_n c_n z^n$ has infinitely many zeros. We then use this result to give alternative proofs that the Le Roy functions $f_r(z)=\sum_{n=0}^\infty \frac{z^n}{(n!)^r}$ for $r>1$ and Bessel functions $J_\alpha(z)=\sum_{m=0}^\infty \frac{(-1)^m}{m!\Gamma(m+\alpha+1)}\left(\frac{z}{2}\right)^{2m+\alpha}$ for $\alpha\in\R$ have infinitely many zeros.
	\end{abstract}

	
	\maketitle


    \section{Main Result}
    The following theorem is the main result of this paper:
    \begin{theorem}\label{thm1.1}
        Suppose $\sum_{n=0}^\infty n!c_n z^n$ is entire and the coefficients $c_n$ do not terminate. Then the function $f(z)=\sum_{n=0}^\infty c_n z^n$ is entire and has infinitely many zeros.
    \end{theorem}
    This gives a simple method for determining if a power series $\sum_n c_n z^n$ has infinitely many zeros. To illustrate its applicability, we give in Section \ref{sec3} alternative proofs that the Le Roy functions\footnote{For more information on the Le Roy functions, the reader is invited to have a look at \cite{le-roy}}
    $$f_r(z)=\sum_{n=0}^\infty \frac{z^n}{(n!)^r},\text{ }r>1$$
    and Bessel functions
    $$J_\alpha(z)=\sum_{m=0}^\infty \frac{(-1)^m}{m!\Gamma(m+\alpha+1)}\left(\frac{z}{2}\right)^{2m+\alpha}$$
    have infinitely many zeros.

    To prove Theorem \ref{thm1.1}, we will need two lemmas:
    \begin{lemma}\label{lem1.2}
        Let $f$ be an entire function of order $1$ with finitely many zeros. Let $N$ be the number of zeros counting multiplicity. Then there exists a constant $k\neq 0$ and a degree-$N$ polynomial $Q(n)$ (depending on $k$) such that for $n>N$, we have $f^{(n)}(0)=k^nQ(n)$.
    \end{lemma}
    \begin{lemma}\label{lem1.3}
        Suppose that $\sum_{n=0}^\infty n!c_n z^n$ converges on some neighborhood of $0$. Then $f(z)=\sum_{n=0}^\infty c_n z^n$ is entire and has order at most $1$.
    \end{lemma}
    Before proving these lemmas, let us show how they can be applied to give the proof of Theorem \ref{thm1.1}:

    \begin{proof}
        Since $\sum_{n=0}^\infty n!c_n z^n$ is entire and \textit{a fortiori} convergent on some neighborhood of $0$, we have by Lemma \ref{lem1.3} that $f$ is entire and has order at most $1$. Suppose, for the sake of reaching a contradiction, that $f$ has \textit{finitely many} zeros. Then the Hadamard factorization of $f$ is
        $$f(z)=e^{A(z)}P(z)$$
        where $A(z)$ and $P(z)$ are polynomials with $A(z)$ either constant or linear. Note that $A(z)$ cannot be constant, otherwise $f$ is a polynomial and hence $c_n=f^{(n)}(0)/n!$ terminates, contradicting our assumption. Thus, $A(z)=l+kz$ where $k\neq 0$, and hence
        $$f(z)=e^{l+kz}P(z)=e^{kz}\tilde{P}(z)$$
        where $\tilde{P}(z):=e^l P(z)$ is a polynomial. Let $N$ be the degree of $\tilde P$, so $N$ is also the number of zeros of $f$ counting multiplicity. By Lemma \ref{lem1.2}, there is a degree-$N$ polynomial $Q(n)$ such that $f^{(n)}(0)=k^n Q(n)$ holds for all $n>N$, so we can write
        $$\sum_{n=N+1}^\infty n!c_n z^n = \sum_{n=N+1}^\infty f^{(n)}(0) z^n=\sum_{n=N+1}^\infty Q(n) (kz)^n$$
        Noting that $Q(n)$ is a polynomial and thus $Q(n)\neq 0$ for all $n$ sufficiently large, we may determine the convergence radius of $\sum_{n=N+1}^\infty Q(n) (kz)^n$ using the ratio test:
        $$\lim_{n\to\infty}\left|\frac{Q(n+1) (kz)^{n+1}}{Q(n) (kz)^n}\right|=|kz|\lim_{n\to\infty}\left|\frac{Q(n+1)}{Q(n)}\right|=|k||z|$$
        To get the last equality, we used $\lim_{n\to\infty}\frac{Q(n+1)}{Q(n)}=1$, which holds because $Q(n+1)$ and $Q(n)$ have the same leading term and coefficient.
        
        From the above, we see that $\sum_{n=N+1}^\infty n! c_n z^n$ diverges whenever $|z|>1/|k|$. But the whole series $\sum_{n=0}^\infty n! c_n z^n$ is entire by hypothesis, so the same must true of the tail $\sum_{n=N+1}^\infty n!c_n z^n$. We have reached a contradiction, so $f$ must have infinitely many zeros.
    \end{proof}
    \begin{remark}
        Since the contradiction obtained in the proof of Theorem \ref{thm1.1} can be reached once we know $\sum_n n! c_n z^n$ converges at some $z_0\in\C$ with $|z_0|>1/|k|$, it is natural to ask whether we can determine the constant $k$ \textit{a priori}, that is, directly from the coefficients $c_n$. This would let us weaken the assumption that $\sum_n n! c_n z^n$ is entire to merely that it has radius of convergence $R>1/|k|$.
        
        Unfortunately, the utility of such a weakening is limited because Theorem \ref{thm1.1} can fail if $\sum_{n=0}^\infty n! c_n z^n$ is not entire. For any possible radius of convergence $0<R<\infty$, consider the series $\sum_{n=0}^\infty n! c_n z^n$ with the choice of Maclaurin coefficients
        $$c_n=\frac{1}{n! R^n}$$
        We can see that $c_n$ never terminates and the series
        $$\sum_{n=0}^\infty n! c_n z^n=\sum_{n=0}^\infty \frac{z^n}{R^n}$$
        has radius of convergence equal to $R$, so it is not entire. However, the corresponding $f$ is
        $$f(z)=\sum_{n=0}^\infty c_n z^n =\sum_{n=0}^\infty \frac{(z/R)^n}{n!}=e^{\frac{z}{R}}$$
        which does not have any zeros.
    \end{remark}
    \section{Proofs of Lemmas 1.2 and 1.3}
    In this section, we prove Lemmas \ref{lem1.2} and \ref{lem1.3}. We first prove \ref{lem1.2}, which is an easy consequence of the Hadamard factorization theorem and the generalized product rule
    $$(fg)^{(n)}(a)=\sum_{j=0}^n {\binom n j}f^{(j)}(a)g^{(n-j)}(a)$$
    for functions $f,g$ that are each $n$-times differentiable at $a$.
    \begin{proof}
        Since $f$ has order $1$ and $N$ zeros counting multiplicity, the Hadamard factorization of $f$ is $f(z)=e^{kz}P(z)$ for some constant $k\neq 0$, where $P(z)$ is a degree-$N$ polynomial. From the generalized Leibniz product rule,
        $$f^{(n)}(0)=\sum_{j=0}^n {\binom n j}k^{n-j}P^{(j)}(0)=k^n\sum_{j=0}^n {\binom n j}\frac{P^{(j)}(0)}{k^j}$$
        If $n>N$, we have $P^{(j)}(0)=0$ since $P$ is a polynomial of degree $N$, so
        $$f^{(n)}(0)=k^n\sum_{j=0}^N {\binom n j}\frac{P^{(j)}(0)}{k^j}\text{ for $n>N$}$$
        Write
        $$P(z)=\sum_{k=0}^N a_k z^k, a_k\in\C$$
        From the formula $a_j=\frac{P^{(j)}(0)}{j!}$ for the Maclaurin coefficients, we see that $P^{(j)}(0)=j!a_j$, so
        \begin{align*}
            f^{(n)}(0) &= k^n\sum_{j=0}^N {\binom n j}\frac{j!a_j}{k^j}\\
            &= k^n\sum_{j=0}^N \frac{n!}{(n-j)!}\cdot\frac{a_j}{k^j}\\
            &= k^n\sum_{j=0}^N \frac{a_j}{k^j}n(n-1)(n-2)\cdots(n-(j-1))
        \end{align*}
        The expressions
        $$\frac{a_j}{k^j}n(n-1)(n-2)\cdots(n-(j-1))$$
        are either $0$ or a polynomial (in $n$) of degree $j$, depending on whether $a_j=0$ or $a_j\neq 0$. Since $a_N\neq 0$ because $P$ has degree $N$, and since
        $$Q(n):=\sum_{j=0}^N \frac{a_j}{k^j}n(n-1)(n-2)\cdots(n-(j-1))$$
        is a finite sum of the above expressions, it must be a degree $N$ polynomial. This completes the proof of Lemma \ref{lem1.2}
    \end{proof}
    We now prove Lemma \ref{lem1.3}.
    \begin{proof}
        Suppose $\sum_{n=0}^\infty n!c_n z^n$ converges for $|z|<\delta$, where $\delta>0$. Evaluating at $z=\frac{\delta}{2}$ gives the convergent series $\sum_{n=0}^\infty n!c_n \left(\frac{\delta}{2}\right)^n$, so the terms $n!c_n \left(\frac{\delta}{2}\right)^n$ go to $0$ and are thus bounded. Suppose $M>0$ is such that $\left|n!c_n \left(\frac{\delta}{2}\right)^n\right|\leq M$. Then for all $z\in\C$,
        $$|c_nz^n|\leq \frac{M|z|^n\left(\frac{2}{\delta}\right)^n}{n!}=\frac{M\left(\frac{2|z|}{\delta}\right)^n}{n!}$$
        Summing the right side over $n\geq 0$ gives a series which converges to $Me^{2|z|/\delta}$, so $f$ is entire by comparison. We also see that
        \begin{equation}\label{ineq1.1}
            |f(z)|\leq \sum_{n=0}^\infty |c_n z^n|\leq M\sum_{n=0}^\infty\frac{\left(\frac{2|z|}{\delta}\right)^n}{n!}=Me^{2|z|/\delta}
        \end{equation}
        so $f$ has order at most $1$.
    \end{proof}
    \section{Applications of Theorem 1.1}\label{sec3}
   As promised, we now give alternative proofs that the Le Roy functions
   $$f_r(z)=\sum_{n=0}^\infty \frac{z^n}{(n!)^r}, r>1$$
   and Bessel functions
   $$J_\alpha(z)=\sum_{m=0}^\infty \frac{(-1)^m}{m!\Gamma(m+\alpha+1)}\left(\frac{z}{2}\right)^{2m+\alpha}$$
   have infinitely many zeros.
   For the Le Roy functions, we obtain a slightly stronger result: there are infinitely many zeros if $\on{Re}(r)>1$.
    \begin{theorem}
        If $\on{Re}(r)>1$, the function $f_r(z)=\sum_{n=0}^\infty \frac{z^n}{(n!)^r}$ is entire and has infinitely many zeros.
    \end{theorem}
    \begin{proof}
        By the ratio test, the series
        $$\sum_{n=0}^\infty n!\cdot \frac{z^n}{(n!)^r}=\sum_{n=0}^\infty \frac{z^n}{(n!)^{r-1}}$$
        is entire if $\on{Re}(r)>1$. The desired result follows from Theorem \ref{thm1.1}.
    \end{proof}
    We now prove that the Bessel functions have infinitely many zeros for any parameter $\alpha\in\R$.
    \begin{theorem}
        For each $\alpha\in\R$, the Bessel function
        $$J_\alpha(z)=\sum_{m=0}^\infty \frac{(-1)^m}{m!\Gamma(m+\alpha+1)}\left(\frac{z}{2}\right)^{2m+\alpha}$$
        has infinitely many zeros.
    \end{theorem}
    \begin{proof}
        Fix $\alpha\in\R$ and define
        $$c_m:=\frac{(-1)^m}{m!\Gamma(m+\alpha+1)2^{2m}}$$
        The series
        $$\sum_{m=0}^\infty m!c_m z^m=\sum_{m=0}^\infty \frac{(-1)^m}{\Gamma(m+\alpha+1)2^{2m}}z^m$$
        is entire by the ratio test, so $g(z):=\sum_{m=0}^\infty c_m z^m$ is entire and has infinitely many zeros by Theorem \ref{thm1.1}.
    
        Now,
        \begin{align*}
            J_\alpha(z) &= \sum_{m=0}^\infty \frac{(-1)^m}{m!\Gamma(m+\alpha+1)}\left(\frac{z}{2}\right)^{2m+\alpha}\\
            &= \left(\frac{z}{2}\right)^\alpha\sum_{m=0}^\infty \frac{(-1)^m}{m!\Gamma(m+\alpha+1)2^{2m}}(z^2)^m\\
            &= \left(\frac{z}{2}\right)^\alpha g(z^2)
        \end{align*}
        Let $a_1,a_2,\dots$ be the (infinitely many) zeros of $g(z)$. Then their principal square roots
        $$\sqrt{a_j}:=|a_j|^{\frac{1}{2}}e^{\frac{i\operatorname{Arg}(a_i)}{2}}$$
        are zeros of $g(z^2)$ and therefore also of $J_\alpha (z)$. It follows that $J_\alpha(z)$ also has infinitely many zeros.
    \end{proof}
    
	\bibliographystyle{plain} 
	\bibliography{refs} 

\end{document}